\documentclass{amsart}

\usepackage[utf8x]{inputenc}

\usepackage{graphicx}
\usepackage[shortlabels]{enumitem}
\usepackage{verbatim}
\usepackage{amssymb}
\usepackage{mathrsfs}  
\usepackage{mathabx}
\usepackage{nicefrac}
\usepackage{stmaryrd}

\usepackage{color}

\usepackage{tikz}
\usetikzlibrary{arrows,chains,matrix,positioning,scopes}
\usetikzlibrary{arrows.meta}
\usetikzlibrary{matrix}

\usepackage{hyperref}
\usepackage {cleveref}

\DeclareMathOperator{\frob}{Frob}
\DeclareMathOperator{\ord}{ord}
\DeclareMathOperator{\Li}{Li}

\newtheorem{theorem}{Theorem}[section]
\newtheorem{proposition}[theorem]{Proposition}
\newtheorem{lemma}[theorem]{Lemma}
\newtheorem{corollary}[theorem]{Corollary}

\theoremstyle{definition}
\newtheorem{definition}[theorem]{Definition}

\theoremstyle{remark}

\numberwithin{equation}{section}

\begin{document}

\title{Rational $2$-Functions are abelian}

\author[L.F. M\"{u}ller]{L. Felipe M\"{u}ller}
\address{Mathematisches Institut, Universit\"{a}t Heidelberg, Im Neuenheimer Feld 205, 69120 Heidelberg, Germany}
\email[corresponding author]{lmueller@mathi.uni-heidelberg.de}


\date{\today.}


\begin{abstract}

We show that the coefficients of rational $2$-functions are contained in an abelian number field.
More precisely, we show that the poles of such functions are poles of order one and  given by roots of unity and rational residue.
\end{abstract}

\maketitle
\tableofcontents

\section{Introduction}

Fermat's and Euler's congruences are well-known in number theory and are rich of remarkable consequences.
In the following we will give a short survey of these congruences.
We start with the famous

\begin{theorem}[Euler]\label{Euler's theorem}
The congruence
\begin{align}\label{euler congruence for integers}
	a^{p^r}\equiv a^{p^{r-1}}\mod p^r
\end{align}
holds for all integers $a \in \mathbb{Z}$, all primes $p$, and all natural numbers $r \in \mathbb{N}$.
\end{theorem}

A sequence $(a_k)_{k \in \mathbb{N}}$ of rational numbers is called an \textit{Euler sequence} (or \textit{Gauss sequence} as in \cite{beu17}) for the prime $p$, if $a_k$ is a $p$-adic integer for all $k \in \mathbb{N}$ and
\begin{align}\label{Euler sequence}
	a_{mp^r}\equiv a_{mp^{r-1}} \mod p^r
\end{align}
for all integers $r\geq 1$ and $m\geq 1$.
A survey of these congruences  has been given in \cite{min14} and \cite{zar08}.

Beukers coined the term \textit{supercongruence}:
A supercongruence (with respect to a prime $p$) refers to a sequence $(a_n)_{n \in \mathbb{N}} \in \mathbb{Z}^{\mathbb{N}}_p$ that satisfies congruences of the type
\begin{align}\label{supercongruence}
	a_{mp^r}\equiv a_{mp^{r-1}} \mod p^{sr},
\end{align}
for all $m,r\in \mathbb{N}$ and a fixed $s\in \mathbb{N}$, $s>1$ (cf. \cite{cos88}).
Such supercongruences are given by the Jacobsthal-Kazandzidis congruence (cf. \cite{bru52}),  Apéry numbers (cf. \cite{ape79}, \cite{beu85}), generalized Domb numbers (cf. \cite{osb12})
and Almkvist-Zudilin numbers (cf. \cite{alm06}, \cite{gor18})
to name a few.
Note that all the above mentioned supercongruences are valid for $s=3$ with respect to $p\geq 5$.

Let $K$ be an algebraic number field and $\mathcal{O}$ its ring of algebraic integers.
We consider a generalization of supercongruences to sequences of algebraic integers in $K$.
More precisely, for $s\in \mathbb{N}$, an $s$-sequence is a sequence $(a_n)\in K^\mathbb{N}$, such that for any unramified prime ideal $\mathfrak{p}\in \mathcal{O}$ lying above the prime $p\in \mathbb{Z}$, $a_n\in \mathcal{O}_\mathfrak{p}$, and for all $m, r\in \mathbb{N}$,
\begin{align*}
\frob_\mathfrak{p} (a_{p^{r-1}m}) - a_{p^rm} \equiv 0 \mod \mathfrak{p}^{sr}\mathcal{O}_\mathfrak{p},
\end{align*}
where $\mathcal{O}_\mathfrak{p}$ is the ring of $\mathfrak{p}$-adic integers and $\frob_{\mathfrak{p}}$ is the canonical lift of the standard Frobenius element of $\mathfrak{p}$ in the Galois group of the local field extension $(\mathcal{O}/ \mathfrak{p}) | (\mathbb{Z}/p)$.
The generating function $V(z)$ of an $s$-sequence then integrates to what is referred to as an $s$-function in \cite{walcher16}.
More precisely, the $s$-sequence $a \in K^{\mathbb{N}}$ corresponds to the $s$-function $\smallint\hspace{-0.25em}^s V(z)$ (see \Cref{s-sequence/function correspondence} in Section 2) given by the (formal) power series
\begin{align*}
\smallint\hspace{-0.25em}^s V(z)=\sum_{n=1}^\infty  \frac{a_n}{n^s}z^n \in z K\llbracket z \rrbracket,
\end{align*}
Interestingly, $2$-functions (where $s=2$) have their geometric origin in super symmetry.
As stated in \cite{walcher16}, see Thm. 22 therein, $2$-functions appear as the non-singular part of the superpotential function (without the constant term) with algebraic coefficients.
In other words, algebraic cycles on Calabi-Yau three-folds provide a source of $2$-functions that are analytic and furthermore satisfy a differential equation with algebraic coefficients.
It is therefore expected that understanding the numerical interpretation of open Gromov-Witten/BPS theory relative to Lagrangian submanifolds  mirror to algebraic cycles highly depend on delivering some (natural) basis of the class of $2$-functions with algebraic coefficients.
It is therefore of main interest to characterize a submodule of $s$-functions by suitable algebraic or analytic properties, and a class of distinguished generators for this submodule.
The contribution of the present work to this problem is to give a characterization of a $2$-function $\smallint^{2} V(z)$, where $V$ represents a rational function.
We have

\begin{theorem}\label{main theorem}
Let $V \in zK\llbracket z \rrbracket$, $V(z)\neq 0$, be the generating function of a $2$-sequence $(a_n)_{n\in \mathbb{N}} \in K^\mathbb{N}$, representing the rational function $F(z)\in K(z)$ as its Maclaurin expansion.
Then, there are rational coefficients $A_i \in \mathbb{Q}$ for $i=1,...,N$ and an appropriate primitive $N$-th root of unity $\zeta$, such that
\begin{align}\label{eq: general rational 2-function}
	F(z)=\sum_{i=1}^N\frac{A_i\zeta ^i z}{1- \zeta^i z}.
\end{align}
In particular, the coefficients $a_n$ of $V(z)$ have the form
\begin{align}\label{explicit coefficients}
	a_n = \sum_{i=1}^N A_i \zeta^{in}.
\end{align}
\end{theorem}

The first reduction in the proof of \Cref{main theorem} is given by \Cref{rational 1-functions}, a statement due to Minton (cf. \cite{min14}).
It states, that the generating functions of Euler sequences are given by sums of logarithmic derivatives of polynomials with integral coefficients.

\subsection*{Acknowledgment}
The author is grateful to Johannes Walcher for providing the initial motivation for this work.

\subsection*{Notation}
Throughout this paper, the natural numbers will be meant to be the set of all positive integers, $\mathbb{N}= \{ 1, 2, . . .\}$, while $\mathbb{N}_0 = \mathbb{N}\cup \{0\}$.
If $X$ is a set, then $X^{\mathbb{N}}$ denotes the set of all sequences indexed by the natural numbers, $(x_n)_{n \in \mathbb{N}}\in X^{\mathbb{N}}$.
For a ring $R$ let $R\llbracket z \rrbracket$ denote the ring of formal power series in the variable $z$ with coefficients in $R$.

\section{Preliminaries}\label{Preliminaries}

In this section, we introduce the definitions and notational conventions that will
be used throughout the paper.
We mainly follow the conventions given in \cite{walcher16}.
\\

Let $K$ be a fixed algebraic number field.
Denote by $\mathcal{O}$ the ring of integers of $K$.
Let $D$ be the discriminant of $K | \mathbb{Q}$.
We say that a prime $p\in \mathbb{Z}$ is unramified in $K| \mathbb{Q}$ if all prime ideals $\mathfrak{p} \mid p\mathcal{O}$  are unramified.
Note that an unramified prime $p$ is characterized by the property that $p\nmid D$.
For any prime ideal $\mathfrak{p}$, $\mathcal{O}_\mathfrak{p}$ denotes the ring of $\mathfrak{p}$-adic integers.
Then $\mathcal{O}_{\mathfrak{p}}$ is an integral domain and its field of fractions $K_{\mathfrak{p}} = \mathrm{Quot}(\mathcal{O}_{\mathfrak{p}})$ is the $\mathfrak{p}$-adic completion of $K$.
For an unramified prime $p$, we set $\mathcal{O}_p$ to be given by
\begin{align*}
	\mathcal{O}_p = \prod_{\mathfrak{p} \mid (p)} \mathcal{O}_\mathfrak{p}.
\end{align*}
Analogously,
\begin{align*}
	K_p=\prod_{\mathfrak{p} \mid (p)} K_\mathfrak{p}.
\end{align*}
Multiplication is realized by component-wise multiplication, that is, for $(x_\mathfrak{p})_{\mathfrak{p}\mid (p)}$ and $(y_{\mathfrak{p}})_{\mathfrak{p}\mid (p)} \in \mathcal{O}_p$ (resp. $K_p$) we have $(x_\mathfrak{p})_{\mathfrak{p}\mid (p)}\cdot (y_{\mathfrak{p}})_{\mathfrak{p}\mid (p)}= (x_\mathfrak{p}\cdot y_{\mathfrak{p}})_{\mathfrak{p}\mid (p)}\in \mathcal{O}_p$ ($K_p$ resp.).
Therefore, $K_p$ is a $K$-algebra.
Let $\iota_\mathfrak{p}\colon K \hookrightarrow K_{\mathfrak{p}}$ be the canonical embeddings of $K$ into its $\mathfrak{p}$-adic completion, then $K$ is embedded in $K_p$ by the map $\iota_p \colon K\hookrightarrow K_p$, $x \mapsto (\iota_{\mathfrak{p}}(x))_{\mathfrak{p}\mid (p)}$.
Nonetheless, if it is clear from the context, we will use the same symbol $x$ for $\iota_p(x)$ or $\iota_\mathfrak{p}(x)$, whenever $x\in K$.
We say that $x \in K$ is a \textit{$p$-adic integer} (\textit{$p$-adic unit} resp.), if $x \in \mathcal{O}_{\mathfrak{p}}$ ($x \in \mathcal{O}_{\mathfrak{p}} ^\times$ resp.) with respect to all prime ideals $\mathfrak{p}\mid (p)$.

$\iota_\mathfrak{p}$ ($\iota_p$ resp.) can be extended to the ring of formal power series $K_\mathfrak{p}\llbracket z\rrbracket$ ($K_p \llbracket z \rrbracket$ resp.) by setting $\iota_\mathfrak{p}(z)=z$ and $\iota_p (z)=z$ and linear extending to maps $\iota_\mathfrak{p}\colon K\llbracket z \rrbracket \hookrightarrow K_\mathfrak{p}\llbracket z\rrbracket$ and $\iota_p\colon K \llbracket z \rrbracket \hookrightarrow K_p \llbracket z \rrbracket$.
Again, for $V\in K\llbracket z \rrbracket$, we will use the same symbol $V$ to refer to the power series $\iota_\mathfrak{p}(V) \in K_\mathfrak{p} \llbracket z \rrbracket$ and $\iota_p (V) \in K_p \llbracket z \rrbracket$.

For $\mathfrak{p}\mid (p)$, the \textit{Frobenius element} $\mathrm{Fr}_{\mathfrak{p}}$ at $\mathfrak{p}$ is the unique element satisfying the following two conditions:
$\mathrm{Fr}_{\mathfrak{p}}$ is an element in the decomposition group $D(\mathfrak{p}) \subset \mathrm{Gal}(K/\mathbb{Q})$ of $\mathfrak{p}$ and for all $x \in \mathcal{O}$, $\mathrm{Fr}_{\mathfrak{p}}(x)\equiv x^p \mod \mathfrak{p}$.
By Hensel's Lemma, $\mathrm{Fr}_{\mathfrak{p}}$ can be lifted to $\mathcal{O}_{\mathfrak{p}}$ and then extended to an automorphism $\mathrm{Frob}_\mathfrak{p}\colon K_\mathfrak{p}\rightarrow K_{\mathfrak{p}}$.
We obtain an isomorphism, denoted by $\mathrm{Frob}_p \colon K_p \rightarrow K_p$, and given by $\mathrm{Frob}_p = (\mathrm{Frob}_ {\mathfrak{p}} )_ {\mathfrak{p} \mid (p)}$.
By declaring $\mathrm{Frob}_p(z)=z$, $\mathrm{Frob}_p$ can be (linearly) extended to an endomorphism $\mathrm{Frob}_p\colon K_p\llbracket z\rrbracket \rightarrow K_p \llbracket z\rrbracket$.

In the following, let $R$ be a $\mathbb{Q}$-algebra.
The \textit{logarithmic derivative} $\delta_R$ is an operator $\delta_R \colon R \llbracket z \rrbracket \rightarrow R \llbracket z \rrbracket$, where $\delta_R$ is given by $z\frac{\mathrm{d}}{\mathrm{d}z}$, i.e.
\begin{align*}
	\delta_R \left[ \sum_{n=0}^\infty r_n z^n\right]
	= \sum_{n=0}^\infty n r_n z^n.
\end{align*}
Its (partial) inverse of $\delta_R$ is the \textit{logarithmic integration} $\smallint_R \colon z R \llbracket z \rrbracket \rightarrow z R \llbracket z \rrbracket$ given by
\begin{align*}
	\smallint\hspace{-0.35em}\,_{R}\left[ \sum_{n=1}^\infty r_n z^n \right]
	= \sum_{n=1}^\infty \frac{r_n} n z^n\qquad\text{and}\qquad \smallint\hspace{-0.35em}\,_R(0)=0.
\end{align*}
For a number $k\in \mathbb{N}$ let $\mathscr{C}_{R,k}$ be the operator $\mathscr{C}_{R,k}\colon R \llbracket z \rrbracket \rightarrow R \llbracket z \rrbracket$, called the \textit{Cartier operator}, given by
\begin{align*}
	\mathscr{C}_{R,k}\left[\sum_{n=0}^\infty r_nz^n\right]
	=\sum_{n=0}^\infty r_{kn}z^n.
\end{align*}
For a number $\ell \in \mathbb{N}$, let $\varepsilon_{R,\ell} \colon R \llbracket z \rrbracket \rightarrow R \llbracket z \rrbracket$ be the $R$-algebra homomorphism uniquely determined by setting
\begin{align*}\varepsilon_{R,\ell} (z) = z^\ell.
\end{align*}
Hereafter, we will omit $R$ from the notation of $\delta_R$, $\smallint_R$, $\mathscr{C}_{R,k}$ and $\varepsilon_{R,\ell}$.
In \cite{walcher16}, an \textit{$s$-function with coefficients in $K$} (for $s \in \mathbb{N}$) is defined to be a formal power series $\widetilde V \in z K\llbracket z \rrbracket$ such that for every unramified prime $p\in \mathbb{Z}$ in $K | \mathbb{Q}$ we have
\begin{align}\label{eq: s-function}
	\frac1{p^s}\mathrm{Frob}_p \widetilde V(z^p) - \widetilde V(z)
	\in z \mathcal{O}_p \llbracket z \rrbracket.
\end{align}
A sequence $(a_n)_{n\in \mathbb{N}}\in K^{\mathbb{N}}$ is said to satisfy the \textit{local $s$-function property for $p$}, if $p \in \mathbb{Z}$ is unramified in $K|\mathbb{Q}$, and $a_n \in \mathcal{O}_p$ is a $p$-adic integer for all $n \in \mathbb{N}$, and
\begin{align}\label{eq: s-sequences}
	\mathrm{Frob}_p \left( a_{mp^{r-1}} \right)
	\equiv a_{mp^r} \mod p^{sr}\mathcal{O}_p,
\end{align}
for all $m,r\in \mathbb{N}$. $(a_n)_{n \in \mathbb{N}}$ is called an \textit{$s$-sequence} if it satisfies the local $s$-function property for all unramified primes $p$ in $K|\mathbb{Q}$.
By definition, it is evident that the coefficients of an $s$-sequence are contained in $\mathcal{O}\left[ D^{-1} \right]$.
The coefficients of $s$-functions are given by an $s$-sequence after applying $s$-fold logarithmic derivation, which is the statement of the equivalence \Cref{s-sequence/function correspondence} \textit{(i)} $\Leftrightarrow$ \textit{(ii)}, see also \cite[Lem. 4]{walcher16}.
We denote by $\mathcal{S}^s(K|\mathbb{Q})\subset z \mathcal{O}\left[D^{-1}\right]\llbracket z\rrbracket$ the set of all generating functions of $s$-sequences with coefficients in $K$
\begin{align*}
	\mathcal{S}^s(K|\mathbb{Q}):= \left\{V \in z\mathcal{O}\left[D^{-1}\right]\llbracket z \rrbracket;\, V=\sum_{n=1}^\infty a_n z^n, \text{ where $(a_n)_{n\in\mathbb{N}}$ is an $s$-sequence}\right\}.
\end{align*}
Furthermore, $\overline{\mathcal{S}}^s (K|\mathbb{Q})\subset z K\llbracket z \rrbracket$ denote the set of formal power series which diver from diver of being an element in $\mathcal{S}^s (K|\mathbb{Q})$ by a rational constant, i.e.
\begin{align*}
\overline{\mathcal{S}}^s(K|\mathbb{Q}):= \left\{V \in zK\llbracket z \rrbracket;\,\text{there is a constant $C\in \mathbb{N}$, such that $C V(z)\in \mathcal{S}^s(K|\mathbb{Q})$}\right\}.
\end{align*}
Let $S$ be a finite set consisting of prime numbers, then
\begin{align*}
\mathcal{S}^s(K|\mathbb{Q})_S := \left\{V \in z\mathcal{O}\left[D^{-1}, q^{-1}; q\in S\right]\llbracket z \rrbracket;\, V=\sum_{n=1}^\infty a_n z^n, \text{ where $(a_n)_{n\in\mathbb{N}}$}\right. \\ \left. \vphantom{\sum_{n=1}^\infty}\text{ satisfies the local $s$-function property for all unramified $p\not\in S$ }\right\}.
\end{align*}
Also,
\begin{align*}
\mathcal{S}^s(K|\mathbb{Q})_{\mathrm{fin}} = \bigcup_{S} \mathcal{S}^s(K|\mathbb{Q})_S,
\end{align*}
where $S$ runs through all finite subsets of rational primes.
Analogously, the sets $\overline{\mathcal{S}}^s(K|\mathbb{Q})_S$ and $\overline{\mathcal{S}}^s (K|\mathbb{Q})_{\mathrm{fin}}$ are defined.
Naturally, we obtain the sequence
\begin{align}\label{eq: descending chain s-functions}
	\mathcal{S}^1(K|\mathbb{Q})\supset
	\mathcal{S}^2(K|\mathbb{Q})\supset \cdots \supset
	\mathcal{S}^{s-1}(K|\mathbb{Q})\supset
	\mathcal{S}^{s}(K|\mathbb{Q})\supset
	\mathcal{S}^{s+1}(K|\mathbb{Q})\supset \cdots
\end{align}

\begin{proposition}\label{s-sequence/function correspondence}
Let $s \in \mathbb{N}$. Then the following is equivalent:
\begin{enumerate}[(i)]
\item
$V\in \mathcal{S}^s(K|\mathbb{Q})$,
\item
$\smallint\hspace{-0.25em}\,^s V$ is an $s$-function,
\item
for all unramified primes $p$ in $K|\mathbb{Q}$ and all $r\in \mathbb{N}$,
\begin{align}\label{eq: definition by cartier operator}
\mathscr{C}_p^{r-1} \left( \mathrm{Frob}_p V(z) - \mathscr{C}_p V(z) \right) &\equiv 0 \mod p^{sr}\mathcal{O}_p\llbracket z\rrbracket,\quad \text{and}\nonumber\\ 
V(z) - \varepsilon_p \mathscr{C}_p V(z) &\in z \mathcal{O}_p\llbracket z \rrbracket.
\end{align}

\item
There is a sequence $b\in K^\mathbb{N}$ satisfying
\begin{align*}
\sum_{i=1}^{\ord_p(n)} \frac{\frob_p(b_{n/p^i}) - b_{n/p^i}}{p^{si}} - b_n \in \mathcal{O}_p
\end{align*}
for all $n\in\mathbb{N}$ and unramified $p$ in $K|\mathbb{Q}$, such that $\smallint\hspace{-0.25em}\,^s V(z)$ can be represented as a formal sum of polylogarithms in the following way
\begin{align*}
\smallint\hspace{-0.25em}\,^s(V(z))
= \sum_{n=1}^\infty b_n \Li_s(z^n).
\end{align*}
\item
There is a sequence $q \in \mathcal{O}\left[ D^{-1}\right]^{\mathbb{N}}$ satisfying
\begin{align*}
\sum_{d\mid n} \frac{\frob_p\left( q_{n/d}^d \right) - q_{n/d}^{pd}}d - p \sum_{\substack{d\mid n \\ p\nmid d}} \frac{q_{n/d}^d}d
\equiv 0\mod p^{(s-1)\ord_p(n)+s} \mathcal{O}_p
\end{align*}
for all $n\in \mathbb{N}$ and unramified $p$ in $K|\mathbb{Q}$, such that $\smallint\hspace{-0.25em}^s V(z)$ can be represented as a formal sum of polylogarithms in the following way
\begin{align*}
	\smallint\hspace{-0.25em}\,^s V(z)
	= \sum_{d=1}^\infty \frac1{d^{s-1}} \Li_s(q_d z^d)
\end{align*}
\end{enumerate}
\end{proposition}

\begin{proof}
Write $a_n:= [V(z)]_n$ for all $n\in \mathbb{N}$.
\begin{itemize}
\item[\textit{(i)} $\Leftrightarrow$ \textit{(ii)}:]
Compute
\begin{align*}
\frac1{p^s}\mathrm{Frob}_p\smallint\hspace{-0.25em}\,^s V (z^p) - \smallint \hspace{-0.25em}\,^s V(z)
&= \frac1{p^s} \sum_{n=1}^\infty \frac{\mathrm{Frob}_p (a_n)}{n^s} z^{pn} - \sum_{n=1}^\infty \frac{a_n}{n^s} z^n\\
&= -\sum_{\substack{n=1\\ p\nmid n}}^\infty \frac{a_n}{n^s} z^n + \sum_{n=1}^\infty \frac{\mathrm{Frob}_p (a_n)- a_{pn}}{p^s n^s} z^{pn}
\end{align*}
Note that the $p$-adic integrality of the first sum is not disturbed by the denominators $n^s$, since their $p$-adic oder is $0$.
Therefore, the equivalence $V \in \mathcal{S}^s(K|\mathbb{Q})$ if and only if $\smallint\hspace{-0.25em}\,^s V(z)$ follows immediately.
\item[\textit{(i)} $\Leftrightarrow$ \textit{(iii)}:]
Let $p$ be unramified in $K|\mathbb{Q}$ and $r\in \mathbb{N}$.
Then
\begin{align*}
\mathscr{C}_p^{r-1}\left(\mathrm{Frob}_p V(z) - \mathscr{C}_p V(z)\right)
&= \mathscr{C}_p^{r-1}\left( \sum_{n=1}^\infty (\mathrm{Frob}_p (a_n) - a_{pn}) z^n \right)\\
&= \sum_{n=1}^\infty \left( \mathrm{Frob}_p\left( a_{p^{r-1}n} \right) - a_{p^r n} \right)z^n,
\end{align*}
and
\begin{align*}
V(z)- \varepsilon_p \mathscr{C}_p V(z)= \sum_{\substack{n=1\\ p\nmid n}}^\infty a_n z^n.
\end{align*}
The condition that $V- \varepsilon_p \mathscr{C}_p V \in \mathcal{O}_p\llbracket z \rrbracket$  is then equivalent to saying that every coefficient whose index is not a multiple of $p$ is a $p$-adic integer, $a_n \in \mathcal{O}_p$ for all $p\nmid n$.
Therefore, for all unramified primes $p$ in $K|\mathbb{Q}$ and $r\in \mathbb{N}$,
\begin{align*}
\mathscr{C}_p^{r-1} \left( \mathrm{Frob}_p V(z) - \mathscr{C}_p V(z) \right) &\equiv 0 \mod p^{sr} z \mathcal{O}_p\llbracket z\rrbracket,\quad \text{and}\\
V(z)- \varepsilon_p \mathscr{C}_p V(z) &\in z \mathcal{O}_p\llbracket z \rrbracket
\end{align*}
if and only if $V \in \mathcal{S}^s (K|\mathbb{Q})$.
\item[\textit{(iv) $\Rightarrow$ (i)}:]
Let $b\in K^{\mathbb{N}}$ such that
\begin{align*}
	\sum_{n=1}^\infty \frac{a_n}{n^s}z^n
	=\sum_{n=1}^\infty b_n\Li_s(z^n).
\end{align*}
By comparing coefficients, we can write equivalently for all $n\in \mathbb{N}$,
\begin{align*}
	a_n=n^s\sum_{d \mid n} \frac{b_d}{(n/d)^s}=\sum_{d\mid n} d^s b_d.
\end{align*}
Let us assume for all $n\in \mathbb{N}$ and all unramified  primes $p$ in $K|\mathbb{Q}$ that
\begin{align*}
\sum_{i=1}^{\ord_p(n)} \frac{\frob_p(b_{n/p^i}) - b_{n/p^i}}{p^{si}} - b_n \in \mathcal{O}_p.
\end{align*}
Write $n= mp^r$ for $m, r \in \mathbb{N}$ with $\gcd(p,m)=1$ (i.e. $\ord_p(n)=r$).
We then obtain
\begin{align*}
\frob_p(a_{mp^{r-1}})-a_{mp^r}\hspace{-5em}&\\
&=\sum_{d\mid \nicefrac np} d^s\frob_p(b_d)- \sum_{d\mid n}d^s b_d\\
&=\sum_{i=0}^{r-1}\sum_{d\mid m} (d p^i)^s \frob_p (b_{dp^i}) - \sum_{i=0}^r \sum_{d\mid m} (dp^i)^sb_{dp^i}\\
&=p^{sr} \sum_{d\mid m} d^s \left(\sum_{i=0}^{r-1} p^{(i-r)s} \mathrm{Frob}_p(b_{dp^i}) -\sum_{i=0}^r p^{(i-r)s} b_{dp^i} \right)\\
&= p^{sr} \sum_{d\mid m} d^s \left( \sum_{i=0}^{r-1} p^{-i-1} \left( \mathrm{Frob}_p(b_{dp^{r-i-1}} - b_{dp^{r-i-1}}\right) - b_{dp^r} \right)\\
&=p^{sr}\sum_{d\mid m}d^s \underbrace{\left(\sum_{i=1}^r \frac{ \frob_p(b_{d p^{r-i}} ) - b_{d p^{r-i}}}{p^{si}}-b_{dp^r} \right)}_{\in\mathcal{O}_p}\\
&\equiv 0\mod p^{sr}\mathcal{O}_p.
\end{align*}
Furthermore, if $r=0$, the sum $\displaystyle \sum_{i=1}^{\mathrm{ord}_p(n)} \frac{\mathrm{Frob}_p (b_{n/p^i}) - b_{n/p^i}}{p^{si}}$ is the empty sum, i.e. equals to $0$.
Therefore, $b_n\in \mathcal{O}_p$ whenever $\ord_p(n)=0$.
Consequently, $\displaystyle a_n = \sum_{d|n} d^s b_d \in \mathcal{O}_p$ in that case.
\item[\textit{(i) $\Rightarrow$ (iv)}:]
By the \textit{Möbius inversion formula} we have
\begin{align*}
	b_n=\frac1{n^s}\sum_{d\mid n}\mu\left(
	\frac nd\right) a_d.
\end{align*}
First assume $\ord_p(n)=0$.
Then $b_n$ is the sum of $p$-adic integers and therefore, $b_n$ is itself a $p$-adic integer.
Now we may assume $\ord_p(n)>0$.
Again, write $n= mp^r$ for $m, r \in \mathbb{N}$ with $\gcd (m,p)=1$ (i.e. $\ord_p(n) = r$).
Recall that $\mu(k)\neq 0$ if and only if $k$ is square-free, therefore,
\begin{align*}
	b_n=\frac1{n^s}
	\sum_{d\mid n}\mu\left(\frac nd\right)a_d
	=\frac1{m^sp^{sr}}
	\sum_{d\mid m}\mu\left(\frac md\right)
	(a_{dp^r}-a_{dp^{r-1}}).
\end{align*}
Hence,
\begin{align*}
\sum_{i=1}^r\frac{ \frob_p (b_{\nicefrac n{p^i}}) - b_{\nicefrac n{p^i}}} {p^{si}} -b_n\hspace{-9em}&\\
&= \frac1{m^s} \sum_{d\mid m} \mu \left( \frac md \right) \frac{ \frob_p \left( \displaystyle{\sum_{i=1}^r} (a_{dp^{r-i}} - a_{dp^{r-i-1}} )\right) - \displaystyle{\sum_{i=0}^r} (a_{dp^{r-i}} - a_{dp^{r-i-1}})}{p^{sr}}\\
&=\frac1{m^s} \sum_{d\mid m} \mu \left( \frac md \right) \frac{ \frob_p (a_{dp^{r-1}}) - a_{dp^r}}{p^{sr}}.
\end{align*}
Assuming $V\in \mathcal{S}^{s}(K|\mathbb{Q})$ therefore implies
\begin{align*}
\sum_{i=1}^r\frac{ \frob_p (b_{\nicefrac n{p^i}}) - b_{\nicefrac n{p^i}}} {p^{si}} -b_n \in \mathcal{O}_p.
\end{align*}
\item[\textit{(i)} $\Leftrightarrow$ \textit{(v)}:]
Find a sequence $q \in K^\mathbb{N}$, such that
\begin{align*}
	a_n = \sum_{d\mid n} \frac nd q_{n/d}^d.
\end{align*}
Indeed, $q_n$ can be defined recursively by
\begin{align*}
	q_n = a_n - \sum_{\substack{ d\mid n \\ d>1}} \frac nd q_{n/d}^d.
\end{align*}
Therefore, $q \in \mathcal{O}\left[ D^{-1} \right]^{\mathbb{N}}$ if and only if $a \in \mathcal{O}\left[ D^{-1} \right]^{\mathbb{N}}$.
We obtain
\begin{align*}
	\smallint\hspace{-0.25em}\,^s V(z) = \sum_{n=1}^\infty \frac{a_n}{n^s} z^n
	=\sum_{n=1}^\infty \sum_{d \mid n} \frac{q_{n/d}^d}{n^{s-1}d} z^n.
\end{align*}
By substitution $n \mapsto d m $ we obtain
\begin{align*}
\smallint\hspace{-0.25em}\,^s V(z) &=
\sum_{m=1}^\infty \sum_{d=1}^\infty \frac{q_m^d}{(dm)^{s-1} d} z^{dm}
=\sum_{m=1}^\infty \frac1{m^{s-1}} \sum_{d=1}^\infty \frac{q_m^d}{d^s} z^{dm}\\
	&= \sum_{m=1}^\infty\frac1{m^{s-1}} \Li_s\left( q_m z^m\right).
\end{align*}
Furthermore,
\begin{align*}
	\frob_p(a_n) - a_{np}&=
	n\left[ \sum_{d\mid n} \frac{\frob_p\left( q_{k/d}^d \right)}d - p\sum_{d\mid np} \frac{q_{np/d}^d}{d} \right]\\
	&= n\left[ \sum_{d\mid n} \frac{\frob_p\left( q_{n/d}^d \right) - q_{n/d}^{pd}}d - p \sum_{\substack{d\mid n \\ p\nmid d}} \frac{q_{n/d}^d}d \right].
\end{align*}
Hence, $V\in \mathcal{S}^s(K|\mathbb{Q})$ if and only if
\begin{align*}
	\sum_{d\mid n} \frac{\frob_p\left( q_{n/d}^d \right) - q_{n/d}^{pd}}d - p \sum_{\substack{d\mid n \\ p\nmid d}} \frac{q_{n/d}^d}d
	\equiv 0 \mod p^{(s-1)\ord_p(n)+s}\mathcal{O}_p.
\end{align*}
\end{itemize}
This completes the proof.
\end{proof}

\section{Dwork's Integrality Lemma}

Let us rephrase \textit{Dwork's Integrality Lemma} in the setting of $1$-functions as has been done in \cite{walcher16}:

\begin{theorem}[cf. Prop. 7 in \cite{walcher16}, Dwork's Integrality Lemma]\label{characterization of 1-functions}
Let $V \in zK\llbracket z \rrbracket$ and $Y \in 1 + z K\llbracket z \rrbracket$ be related by $V = \log Y$, $Y = \exp (V)$.
Then the following is equivalent
\begin{enumerate}[(i)]
\item
$V$ is a $1$-function.
\item
 There is a sequence $q\in \mathcal{O}\left[D^{-1}\right]^\mathbb{N}$ such that
 \begin{align*}
 \smallint V(z) = - \sum_{n=1}^\infty \log (1- q_nz^n).
 \end{align*}
\item
For every unramified prime $p$ in $K|\mathbb{Q}$,
\begin{align*}
\frac{\mathrm{Frob}_p(Y)(z^p)}{Y(z)^p} \in 1 + z p \mathcal{O}_p\llbracket z \rrbracket.
\end{align*}
\item
$Y \in 1+ z \mathcal{O}\left[ D^{-1} \right] \llbracket z \rrbracket$.
\end{enumerate}
\end{theorem}

\begin{proof}
Let $p$ be a prime unramified in $K\mid \mathbb{Q}$.
\begin{enumerate}

\item[\textit{(i) $\Leftrightarrow$ (ii)}]
Using the equivalence \Cref{s-sequence/function correspondence} \textit{(ii)} $\Leftrightarrow$ \textit{(v)} for $s=1$, the statement follows from
\begin{align}\label{eq: characterization of 1-functions}
\sum_{d\mid n} \frac{\frob_p\left( q_{n/d}^d \right) - q_{n/d}^{pd}}d - p \sum_{\substack{d\mid n \\ p\nmid d}} \frac{q_{n/d}^d}d
\equiv 0 \mod p\mathcal{O}_p.
\end{align}
If $p\nmid d$, then $p\frac{q_{n/d}^d}d\in p \mathcal{O}_p$. Therefore, \cref{eq: characterization of 1-functions} follows from Euler's Theorem, for all $q_{n/d} \in \mathcal{O} \left[ D^{-1} \right]$,
\begin{align*}
	\frob_p\left( q_{n/d}^d \right) - q_{n/d}^{pd}
	\equiv 0 \mod p^{\ord_p(d)+1} \mathcal{O}_p.
\end{align*}

\item[\textit{(ii) $\Rightarrow$ (iv)}]
Given \textit{(ii)}, we have
\begin{align*}
	Y(z)=\exp(V(z))=\prod_{d=1}^\infty(1-q_dz^d)^{-1}\in
	1+z\mathcal{O}\left[ D^{-1} \right] \llbracket z \rrbracket.
\end{align*}

\item[\textit{(iv) $\Rightarrow$ (iii)}]
Let $y\in \mathcal{O}\left[ D^{-1} \right] ^\mathbb{N}$  be given by the sequence
\begin{align*}
y:= ([Y(z)]_{n-1})_{n\in \mathbb{N}}.
\end{align*}
(The shift $n-1$ in the index above is due to the fact that $Y(z)$ has leading constant (zeroth) coefficient.)
Then $Y(z)^p$ can be expressed in terms of (partial) Bell polynomials as follows.
Using the convention $!y= (n! y_n)_{n\in \mathbb{N}}$, we obtain
\begin{align*}
Y(z)^p = \frac1{z^p}\left( \sum_{n=1}^\infty y_n z^n \right)^p = \sum_{n=p}^\infty \frac{p!}{n!} B_{n,p}(!y) z^{n-p}.
\end{align*}
Furthermore,
\begin{align*}
	\frac{p!}{n!} B_{n,p}(!y) = \sum_{\alpha \in \pi(n,p)} \binom{p}{\alpha_1, ..., \alpha_{n-p+1}} \prod_{i=1}^{n-p+1} y_i^{\alpha_i},
\end{align*}
where $\alpha \in \pi(n,p)\subset \mathbb{N}_0^{n-p+1}$ if and only if
\begin{align*}
	\sum_{i=1}^{n-p+1} \alpha_i = p \qquad \text{and} \qquad \sum_{i=1}^{n-p+1} i\alpha_i = n.
\end{align*}
If there is a $1\leq j \leq n-p +1$ such that $\alpha_j<p$, then
\begin{align*}
\binom{p}{\alpha_1,..., \alpha_{n-p+1}} = \frac p{\alpha_j} \binom{p-1}{\alpha_1, ..., \alpha_{j}-1, ... ,\alpha_{n-p+1}} \equiv 0 \mod p\mathcal{O}_p.
\end{align*}
If $\alpha_j=p$, then $\alpha_i=0$ for all $1\leq i \leq n-p+1$, $i\neq j$.
(Indeed, this follows from the condition $\displaystyle \sum_{i=1}^{n-p+1} \alpha_{i}=p$). 
Hence
\begin{align*}
	n = \sum_{i=1}^{n-p+1} i \alpha_i = j p.
\end{align*}
In particular, if $p\nmid n$, then $\displaystyle \frac{p!}{n!} B_{n,p}(!y)\equiv 0 \mod p \mathcal{O}_p$.
We obtain for $p\mid n$,
\begin{align*}
	\frac{p!}{n!} B_{n,p}(!y) \equiv y_{n/p}^p \mod p\mathcal{O}_p.
\end{align*}
Therefore,
\begin{align*}
	Y(z)^p&=
	\sum_{n=p}^\infty \frac{p!}{n!} B_{n,p}(!y) z^{n-p}
	\equiv \sum_{\substack{n=p \\ p\mid n}}^\infty y_{n/p}^p z^{n-p} \mod p\mathcal{O}_p\llbracket z\rrbracket\\
	&=\sum_{n=1}^\infty y_n^p z^{p(n-1)}
	\equiv \sum_{n=1}^\infty \frob_p(y_n) z^{p(n-1)} \mod p\mathcal{O}_p\llbracket z \rrbracket\\
	&=\frob_p\left[ \sum_{n=1}^\infty [Y(z)]_{n-1} z^{p(n-1)} \right]
	= \frob_p Y(z^p).
\end{align*}
Consequently, there is a $g(z)\in z\mathcal{O}_p\llbracket z \rrbracket$, such that
\begin{align*}
	\frob_pY(z^p)= Y(z)^p + p g(z).
\end{align*}
Hence,
\begin{align*}
	\frac{\frob_pY(z^p)}{Y(z)^p} = 1+ p \frac{g(z)}{Y(z)^p}.
\end{align*}
Since $Y\in 1 + z\mathcal{O}\left[D^{-1}\right]\llbracket z \rrbracket$, $Y$ is invertible in $\mathcal{O}_p\llbracket z\rrbracket$ and therefore
\begin{align*}
	\frac{g(z)}{Y(z)^p} \in z\mathcal{O}_p\llbracket z\rrbracket,
\end{align*}
from which \textit{(iii)} follows.

\item[\textit{(iii) $\Rightarrow$ (i)}]
Given \textit{(iii)}, we have an element $g(z)\in z \mathcal{O}_p \llbracket z  \rrbracket$, such that
\begin{align*}
\frac{\frob_p Y(z^p)}{Y(z)^p}=1 + p g(z).
\end{align*}
Taking the logarithm then gives
\begin{align*}
p\left[ \frac1p \frob_p V(z^p) - V(z) \right] = \log(1 + pg(z)) =\sum_{n=1}^\infty \frac{(-p g(z))^n}n
\in pz\mathcal{O}_p\llbracket z \rrbracket.
\end{align*}
In particular, $V$ is an $1$-function.
\end{enumerate}
This completes the proof.
\end{proof}

Dwork himself used his lemma (stated for $K=\mathbb{Q}$) as a key step to prove his theorem that for an affine hypersurface $H$ over a finite field $\mathbb{F}_q$ the zeta-function $Z(H;X)$ of $H$ in the variable $X$ is a rational function and its logarithmic derivative $\frac{Z'(H; X)}{Z(H;X)}$ is the generating function of the non-negative numbers $(N_n)_{n\in\mathbb{N}}$ of $\mathbb{F}_{q^n}$-points of $H$, i.e. $N_n = |H(\mathbb{F}_{q^n})|$.

\section{A Theorem by Minton}

The next \Cref{rational 1-functions} is a modified version of Theorem 7.1 in \cite{beu17}, which on the other hand is a re-proven statement from \cite{min14}.
It is the starting point for the proof of \Cref{main theorem}.
The crucial point is, that a rational 1-function only admits poles of order 1.
We give a proof for the sake of completeness.
In the course of this, we follow the ideas given in \cite{beu17}.

\begin{theorem}[compare with \cite{beu17}, \cite{min14}]\label{rational 1-functions}
Let $V\in \mathcal{S}_{\mathrm{rat}}^1(K|\mathbb{Q})$ representing the rational function $F(z) \in K(z)$ as its Maclaurin expansion.
Then there is a integer $r\in \mathbb{N}$, distinct algebraic numbers $\alpha_i \in \overline{\mathbb{Q}}^\times$, and $A_i\in \mathbb{Q}^\times$, for $i=1,...,r$, such that $F$ can be written as
\begin{align*}
	F(z)
	=\sum_{i=1}^r \frac{A_i\alpha_iz}{1-\alpha_iz}.
\end{align*}
\end{theorem}

\begin{proof}
Let $F$ be given by the fraction of $P,Q \in K[z]$, $Q\not\equiv 0$, i.e. $F=\frac PQ$.
We may assume that $Q(0)\neq 0$ and $P(0)=0$.
By \cite[Prop. 3.5]{beu17}, we have $\mathrm{deg} (P) \leq \mathrm{deg} (Q)$.
By adding a constant $C\in K$ to $F$ it does not affect the $1$-function condition but we may assume $\deg(P)<\deg(Q)$.
Then, by the \textit{Partial Fraction Decomposition} $\widetilde F=\frac PQ+C$ has the form
\begin{align*}
\widetilde F = \sum_{i=1} ^ r \sum_{j=1}^{m_i} \frac{A_{i , j}}{(1 - \alpha_i z )^j},
\end{align*}
where the $\alpha_i\in \overline{\mathbb{Q}}^\times$, $i\in \{1,...,r\}$ are distinct algebraic numbers, $m_i\in\mathbb{N}$  and $A_{i,j}\in\overline{\mathbb{Q}}$ for all $(i,j) \in \{1,...,r\} \times \{1,...,m_i\}$.
Now, let $p$ be a sufficiently large prime, unramified in $K|\mathbb{Q}$, such that the following conditions are simultaneously satisfied:
\begin{enumerate}[(i)]
\item $\alpha_i$ is a $p$-adic unit for all $i \in \{ 1, ... , r \}$,
\item $\alpha_i - \alpha_j$ is a $p$-adic unit for all $i,j \in \{1,...,r\}$, $i\neq j$, and
\item $p>m_i$ for all $i \in \{1,...,r\}$.
\end{enumerate}
What we need to show is $m_i=1$ for all $i \in \{1,...,r\}$.
We have
\begin{align*}
	\frac1{(1 - \alpha_iz)^j}=\sum_{k=0}^\infty \binom{k+j-1}{j-1} \alpha_i^k z^k.
\end{align*}
Therefore, if $\widetilde V(z) = V(z)+C$ is the Maclaurin series expansion of $\widetilde F$, we have
\begin{align*}
\mathscr{C}_p \widetilde V = \sum_{k=0}^\infty \left[ \sum_{i=1}^r \sum_{j=1}^{m_i} A_{i,j} \binom{p k + j - 1}{j - 1} \alpha_i^{pk} \right] z^k.
\end{align*}
Since $p>m_i$, we find $\binom{pk+\nu}{\nu}\equiv 1\mod p$ for all $0\leq \nu<m_i$ (in particular, $\nu<p$) by the following calculation
\begin{align*}
	\binom{pk+\nu}{\nu}=
		\prod_{\ell=1}^\nu\left(1+\frac{pk}\ell\right)
		\equiv 1\mod p.
\end{align*}
Consequently,
\begin{align*}
	\mathscr{C}_p \widetilde V &\equiv
	\sum_{k=0}^\infty \left[
	\sum_{i=1}^r\sum_{j=1}^{m_i}A_{i,j} \alpha^{pk}
	\right]z^k \mod p\\
	&=\sum_{i=1}^r
	\sum_{j=1}^{m_i}\frac{A_{i,j}}{1 - \alpha_i^p z}
	=\sum_{i=1}^r\frac{A_i}{1 - \alpha_i^p z},
\end{align*}
where $A_i=\sum_{j=1}^{m_i}A_{i,j}$.
Hence, $\mathscr{C}_p \widetilde V$ represents a rational function with exclusively simple poles modulo $p$. 
Thus, the $1$-function property $\mathscr{C}_p \widetilde V - \widetilde V \equiv 0 \mod p\mathcal{O}_p\llbracket z \rrbracket$ ensures that $\widetilde F$ has only simple poles as well.
Therefore, we write from now on
\begin{align*}
\widetilde F= \sum_{i=1}^r\frac{A_i}{1 - \alpha_i z},
\end{align*}
where $A_i , \alpha_i \in \overline{\mathbb{Q}}^\times$ and $\alpha_i \neq \alpha_j$ for $i\neq j$.
Evaluating $\tilde F$ at $z=0$ we conclude that $C=\sum_{i=1}^rA_i$.
Therefore,
\begin{align*}
	F=\widetilde F-C
	= \sum_{i=1}^r\frac{A_i}{1- \alpha_i z} - \sum_{i=1}^r A_i
	=\sum_{i=1}^r\frac{A_i \alpha_i z}{1-\alpha_iz}.
\end{align*}
In particular, we have
\begin{align*}
	a_n=\sum_{i=1}^r A_i \alpha_i^n \qquad\text{for all $n\in\mathbb{N}$}.
\end{align*}
The local $1$-function property for $p$ then gives
\begin{align*}
	0\equiv \mathrm{Frob}_p(a_m)-a_{mp}
	=\sum_{i=1}^r\left(
	\mathrm{Frob}_p(A_i)\mathrm{Frob}_p(\alpha_i^m)- A_i\alpha_i^{mp}
	\right)\mod p\mathcal{O}_p,
\end{align*}
for all $m\in\mathbb{N}$.
Since $\mathrm{Frob}_p$ is given by taking component-wise the $p$-th power modulo $\mathfrak{p}$ for all $\mathfrak{p} \mid (p)$, we conclude
\begin{align*}
	0\equiv \sum_{i=1}^r \left(A_i^p-A_i\right) \alpha_i^{mp} \mod p\mathcal{O}_p,
\end{align*}
for all $m\in\mathbb{N}$.
The Vandermonde type matrix $M$
\begin{align*}
	M=\begin{pmatrix}
	\alpha_1^p& \alpha_1^{2p}& \dots & \alpha_1^{rp}\\
	\alpha_2^p& \alpha_2^{2p}& \dots & \alpha_2^{rp}\\
	\vdots& \vdots & \ddots & \vdots\\
	\alpha_r^p& \alpha_r^{2p}& \dots & \alpha_r^{rp}
	\end{pmatrix}.
\end{align*}
is invertible modulo $p\mathcal{O}_p$. Indeed, its determinant is given by
\begin{align*}
	\det (M) &= \big(\prod_{i=1}^r \alpha_i^p\big)\times
	\prod_{1\leq i <j \leq r}\left(\alpha_j^p-\alpha_i^p\right)\\
	&\equiv \big(\prod_{i=1}^r \alpha_i^p\big)\times
	\prod_{1\leq i <j \leq r}\left(\alpha_j-\alpha_i\right)^p
	\mod p \mathcal{O}_p.
\end{align*}
By assumption (i) and (ii) above, we obtain $\det(M)\in \mathcal{O}_p^\times$.
Hence, $A_i^p\equiv A_i\mod p$ for all $i\in\{1,...,r\}$.
From \textit{Frobenius's Densitiy Theorem}, see for instance \cite{jan73}, it follows that $A_i\in\mathbb{Q}$ for all $i\in\mathbb{N}$.
\end{proof}

Originally (compare \cite{walcher16}), an $s$-function $\smallint\hspace{-0.25em}^s V \in z K\llbracket z \rrbracket$ was called \textit{algebraic} if $Y(z) := \exp(- \smallint V)$ is the Maclaurin series expansion of an algebraic function.
Consequently, a rational $s$-function is an $s$-function $\smallint\hspace{-0.25em}^sV(z)$ such that $Y$ is the Maclaurin expansion of a rational function.
However, in the present work, the set $\mathcal{S}^s_{\mathrm{rat}} (K|\mathbb{Q})$ ($\overline{\mathcal{S}}^s_{\mathrm{rat}} (K|\mathbb{Q})$, resp.) denote the subset in $\mathcal{S}^s(K|\mathbb{Q})$ ($\overline{\mathcal{S}}^s(K|\mathbb{Q})$, resp.) of elements which represent rational functions.
This change in terminology is justified by the following statement.

\begin{proposition}\label{cor: algebraicity of Y assuming rationality of V}
Let $V \in \mathcal{S}^1(K|\mathbb{Q})$ and $Y= \exp(-\smallint V)$. Then $V$ is the series expansion of a rational function if $Y$ is the series expansion of a rational function.
Conversely, if $V$ represents a rational function, then there is an $M \in \mathbb{N}$ such that $Y^M$ is the series expansion of a rational function.
\end{proposition}

\begin{proof}

Let $Y$ be the series expansion of a rational function, then so is $\delta Y$.
Hence, $\frac{\delta Y}Y$ is the series expansion of a rational function.
Consequently, $V=-\frac{\delta Y}Y$ represents a rational function.
Note that this holds even for arbitrary $V \in z K\llbracket z \rrbracket$.
Conversely, let $V$ represent a rational function at zero.
By \Cref{rational 1-functions} (here we use $V \in \mathcal{S}^1 (K|\mathbb{Q})$) there exists a natural number $r \in \mathbb{N}$, and distinct $\alpha_i\in \overline{\mathbb{Q}}^\times$, $A_i\in \mathbb{Q}^\times$, for $i = 1 , ... , r$, such that
\begin{align*}
	V = \sum_{i=1}^r\frac{A_i \alpha_i z}{1-\alpha_i z}
	= \sum_{i=1}^r A_i \sum_{n=1}^\infty \alpha_i^n z^n
	=\sum_{i=1}^r A_i \delta \sum_{n=1}^\infty \frac{\alpha_i^n}n z^n
	= - \sum_{i=1}^r A_i\delta \log(1-\alpha_iz).
\end{align*}
Therefore,
\begin{align*}
	Y = \exp(-\smallint V)
	=\exp\left(\sum_{i=1}^r A_i \log(1-\alpha_iz)\right)
	=\prod_{i=1}^r (1-\alpha_iz)^{A_i}.
\end{align*}
Taking $M \in \mathbb{N}$ to be the least common multiple of the denominators of $A_i$ we find that $Y^M$ is a rational function.
\end{proof}

Proving \Cref{cor: algebraicity of Y assuming rationality of V} does not go without mentioning the following more general and known statements.
In fact, we have some analogue statements of \Cref{cor: algebraicity of Y assuming rationality of V} for algebraic functions given in \Cref{prop: V algebraic if Y algebraic} and \Cref{thm: kassel-reutenauer}.
\Cref{prop: V algebraic if Y algebraic} is a direct consequence of the combined work of Stanley (see \cite{stan80}, 1980) and Harris and Sibuya (see \cite{har85}, 1985) as we will demonstrate immediately.
A (formal) power series $V(z) \in K\llbracket z \rrbracket$ D-finite, if all (formal) derivatives of $V$, $\frac{\mathrm{d}^n}{\mathrm{d}^nz} V(z)$, for all $n\in \mathbb{N}$, span a finite dimensional vector space over $K(z)$.

\begin{proposition}\label{prop: V algebraic if Y algebraic}
Let $V \in K\llbracket z \rrbracket$ and $Y= \exp(-\smallint V) \in 1+ z K \llbracket z \rrbracket$. If $Y$ is the series expansion of an algebraic function, then $V$ is the series expansion of an algebraic function.
\end{proposition}

\begin{proof}
Assume that $Y$ represents an algebraic function, i.e. an algebraic element over the field of rational functions $K(z)$.
Then $\frac1Y$ is also an algebraic function.
By the following theorem due to Stanley, $Y$ and $\frac1Y$ are D-finite.
\begin{theorem}[Thm. 2.1 in \cite{stan80}]\label{thm: algebraic --> d-finite}
If $Y \in K\llbracket z \rrbracket$ is algebraic, then $Y$ is D-finite.
\end{theorem}
Additionally, In \cite{har85}, Harris and Sibuya established the following theorem:
\begin{theorem}[Cor. 1 in \cite{har85}]\label{thm: Y d-finite --> V algebraic}
Let $Y\in K\llbracket z \rrbracket$, $Y\neq 0$, be a power series such that $Y$ and $\frac1Y$ are D-finite. Then the logarithmic derivative $\frac{\delta Y}{Y}$ of $Y$ is algebraic over $K(z)$.
\end{theorem}
The statement then follows by recognizing that $V= - \frac{\delta Y}Y$.
\end{proof}

As pointed out in \cite{gar15} (and also in \cite{kas14}), the converse of \Cref{prop: V algebraic if Y algebraic} is not true:
Take for $Y$ the exponential function $\exp(z)$, which is a transcendental formal power series, then $\frac{\delta Y}Y= z$, which is even a rational function.
However, under an additional assumption on the integrality of the coefficients of $Y$, Kassel and Reuntenauer wrote down a proof of the following theorem in \cite{kas14}, 2014, by using the solution to the \textit{Grothendieck-Katz conjecture}.

\begin{theorem}[Thm. 4.4 in \cite{kas14}]\label{thm: kassel-reutenauer}
If $Y\in \mathbb{Z} \llbracket z \rrbracket$ is a formal power series with integral coefficients such that $\frac{\delta Y}Y$ is algebraic, then $Y$ is algebraic.
\end{theorem}

Note from Dwork's Integrality Lemma (cf. \Cref{characterization of 1-functions}) that the integrality condition on the coefficients of $Y$ in the case where $Y\in 1+ z \mathbb{Z}\llbracket z \rrbracket$ is equivalent to saying that $V = (\pm)\frac{\delta Y}Y \in \mathcal{S}^1(\mathbb{Q})$.
This observation coincides with the proof of \Cref{cor: algebraicity of Y assuming rationality of V}.

\section{Algebraic structures}

Recall the operators $\varepsilon_k$ and $\mathscr{C}_\ell$ from above.
Obviously, $\varepsilon_k$ preserves the integrality property \cref{eq: s-function} of an $s$-function $\smallint\hspace{-0.25em}\,^s V(z)$, that is, $\varepsilon_k \smallint\hspace{-0.25em}\,^sV(z)$ remains an $s$-function.
Therefore, we may define $\varepsilon_k^{(s)}\colon \mathcal{S}^s(K|\mathbb{Q}) \rightarrow \mathcal{S}^s(K|\mathbb{Q})$ by the composition
\begin{align}\label{eq: shift}
	\varepsilon_k^{(s)} \colon \mathcal{S}^s(K|\mathbb{Q})
	\xrightarrow{\smallint\hspace{-0.25em}\,^s} zK\llbracket z \rrbracket
	\xrightarrow{\varepsilon_k} zK \llbracket z\rrbracket
	\xrightarrow{\delta^s} \mathcal{S}^s(K|\mathbb{Q}).
\end{align}
Equivalently, $\varepsilon_k^{(s)}$ is given by $z \mapsto k^s z^k$, i.e. $\varepsilon _k ^{(s)} = k^s \varepsilon _k$.
In particular,
\begin{align*}
\varepsilon_k \colon \overline{\mathcal{S}}^s(K|\mathbb{Q}) \rightarrow \overline{\mathcal{S}}^s(K|\mathbb{Q}),
\end{align*}
i.e. the multiplication by $k^s$ can be omitted.
It is also obvious, that the Cartier operator $\mathscr{C}_\ell$ (for $\ell \in \mathbb{N}$) gives a map $\mathscr{C}_\ell \colon \mathcal{S}^s(K|\mathbb{Q}) \rightarrow \mathcal{S}^s(K|\mathbb{Q})$, compare with \Cref{s-sequence/function correspondence} \textit{(i)} $\Leftrightarrow$ \textit{(iii)}.
Note that that $\varepsilon_k^{(s)}$ and $\mathscr{C}_\ell$ preserve rationality, i.e.
\begin{align}\label{eq: preservation of s-functions}
\varepsilon_k, \mathscr{C}_\ell \colon \overline{\mathcal{S}}_{\mathrm{rat}}^s ( K | \mathbb{Q} ) \rightarrow \overline{\mathcal{S}}_{\mathrm{rat}}^s (K|\mathbb{Q}).
\end{align}
This is obvious for $\varepsilon_k^{(s)}$.
To see that $\mathscr{C}_\ell$ preserves rationality, let $\zeta_{\ell}$ denote a primitive $\ell$-th root of unity and note that the Cartan Operator $\mathscr{C}_\ell \colon K \llbracket z \rrbracket \rightarrow K \llbracket z \rrbracket$  can be represented as
\begin{align}\label{eq: cartan operator}
\mathscr{C}_\ell W(z) = \frac1\ell \sum_{r=1}^\ell W \left( \zeta_\ell^r z^{\nicefrac1\ell} \right),
\end{align}
where $W\in K \llbracket z \rrbracket$ is a formal power series.
We obtain

\begin{proposition}
Let $R$ be the ring given by $R= \mathbb{Q} \left[ \varepsilon_k, \mathscr{C}_\ell\,|\, k,\ell \in \mathbb{N} \right]$.
Then $\overline{\mathcal{S}}^s(K|\mathbb{Q})$ is an $R$-module.
Also, $\overline{\mathcal{S}}^2_{\mathrm{rat}} (K | \mathbb{Q} )$ is a $R$-submodule of $\overline{\mathcal{S}}^s(K|\mathbb{Q})$ for all $s\geq 2$.
\end{proposition}

It is clear, that the $s$-function property \cref{eq: s-function} is not respected by regular multiplication of power series.
However, we find that $\mathcal{S}^s(K|\mathbb{Q})$ is closed under the \textit{Hadamard product} of power series.
Let $V,W\in K\llbracket z \rrbracket$, $tV(z)= \sum_{n=0}^\infty a_n z^n$ and $ W(z)= \sum_{n=0}^\infty b_n z^n$.
Then the Hadamard product $V\odot W$ of $V$ and $W$ is given by the power series
\begin{align*}
	V\odot W(z) = \sum_{n=0}^\infty a_n b_n z^n.
\end{align*}

\begin{proposition}\label{prop: hadamard product and s-functions}
$(\mathcal{S}^s(K|\mathbb{Q}), +, \odot)$ is a $\mathbb{Z}\left[ D ^{-1} \right]$-algebra.
\end{proposition}

\begin{proof}
We only need to show that $V\odot W \in \mathcal{S}^s(K|\mathbb{Q})$, whenever $V,W \in \mathcal{S}^s(K|\mathbb{Q})$.
Let therefore $V(z) = \sum_{n=1}^\infty a_n z^n$ and $W(z) = \sum_{n=1}^\infty b_n z^n$, then
\begin{align*}
\mathrm{Frob}_p(a_n b_n) - a_{pn} b_{pn}\equiv
a_{pn}\left( \mathrm{Frob}_p (b_n)- b_{pn} \right)
\equiv 0 \mod p^{s(\ord_p(n)+1)} \mathcal{O}_p,
\end{align*}
as stated.
\end{proof}

Jungen proved in \cite{jun31} (see also \cite{sch61}), that the Hadamard product of a rational and a algebraic function is algebraic, while the Hadamard product of two rational functions remain rational.
Stanley proved in \cite{stan80} that the Hadamard product of two D-finite functions is D-finite.
As a conclusion to sum up, we have the following statement for $\mathcal{S}^s(K|\mathbb{Q})$.

\begin{proposition}
$(\mathcal{S}^s(K|\mathbb{Q}),+,\odot)$ is an $\mathcal{S}^s_{\mathrm{rat}}(K|\mathbb{Q})$-algebra.
Furthermore, $\mathcal{S}^s_{\mathrm{D\text{-}fin}}(K|\mathbb{Q})$ is an $\mathcal{S}^s_{\mathrm{rat}}(K|\mathbb{Q})$-subalgebra and $\mathcal{S}^s_{\mathrm{alg}}(K|\mathbb{Q})$ is an $\mathcal{S}^s_{\mathrm{rat}}(K|\mathbb{Q})$-submodule of $\mathcal{S}^s(K|\mathbb{Q})$.
\end{proposition}

\section{Proof of \Cref{main theorem}}\label{proof of main theorem}

In the present section we will give a proof of \Cref{main theorem}.
More precisely, we prove

\begin{theorem}\label{main theorem (2)}
Let $V \in \overline{\mathcal{S}}^2_{\mathrm{rat}}(K|\mathbb{Q})_{\mathrm{fin}}$, $V(z)\neq 0$, representing the rational function $F(z)\in K(z)$ as its Maclaurin expansion and write $a_n= [V(z)]_n$, for all $n\in \mathbb{N}$.
Then $V$ is \textit{periodic}, i.e. there is an $N \in \mathbb{N}$ such that
\begin{align*}
	N= \min\{ k \in \mathbb{N}\,|\, a_n= a_{n+ k} \text{ for all $n\in\mathbb{N}$}\}.
\end{align*}
Furthermore, there are rational coefficients $A_i \in \mathbb{Q}$ for $i=1,...,N$ and an appropriate primitive $N$-th root of unity $\zeta$, such that
\begin{align}\label{eq: general rational 2-function}
	F(z)=\sum_{i=1}^N\frac{A_i\zeta ^i z}{1- \zeta^i z}, \quad \text{and $A_1\neq 0$}.
\end{align}
In particular, the coefficients $a_n$ of $V(z)$ have the form
\begin{align}\label{explicit coefficients}
	a_n = \sum_{i=1}^N A_i \zeta^{in}.
\end{align}
Moreover, the map $\pi \colon \overline{\mathcal{S}}_\mathrm{rat}^2 (K|\mathbb{Q})_\mathrm{fin} \rightarrow \mathbb{N}_0$, taking $V \mapsto N$ and $0 \mapsto 0$, is surjective. 
\end{theorem}

From \Cref{main theorem (2)}, \Cref{main theorem} follows easily.
By multiplication with a integral constant, we may assume $V\in \mathcal{S}_{\mathrm{rat}}^2 ( K | \mathbb{Q} )_{\mathrm{fin}}$.
Let $S$ be the finite set of those primes, which ramify in $K|\mathbb{Q}$ and at which $V$ does not satisfy the local $2$-function property.
By substituting $K$ by $K(\zeta_q \, |\, q\in S)$, we might also assume $V \in \mathcal{S}_\mathrm{rat}^2(K|\mathbb{Q})$.
Therefore, let $V \in \mathcal{S}_\mathrm{rat}^2 (K|\mathbb{Q})$.
In particular, $V \in \mathcal{S}^1 ( K | \mathbb{Q} )$ by \cref{eq: descending chain s-functions}, and by \Cref{rational 1-functions}, there is an $r \in \mathbb{N}$, $A_i \in \mathbb{Q}^\times$  and distinct $\alpha_i \in \overline{\mathbb{Q}}^\times$ for $i \in \{ 1,...,r \}$ such that
\begin{align*}
	a_n = \sum_{i=1}^r A_i \alpha_i^n \text{ for all $n \in \mathbb{N}$}.
\end{align*}
In the following, let us assume $\alpha_i \in K$, since we might otherwise substitute $K$ by a normal closure of $K(\alpha_1,...,\alpha_r)$.
As pointed out by Minton in \cite{min14} the \textit{Chebotarëv Density Theorem} implies

\begin{theorem}[Thm. 3.3. in \cite{min14}]\label{chebotarev}
	Let $K$ be a Galois number field. For any $\sigma \in 
	\operatorname{Gal}(K | \mathbb{Q})$, there exists infinitely many
	primes $\mathfrak{p}$ of $K$ such that $\mathrm{Fr}_\mathfrak{p}=\sigma$.
\end{theorem}

Let $p\in \mathbb{Z}$ be an unramified prime in $K|\mathbb{Q}$, splitting 
completely in $K$, i.e. $\mathrm{Fr}_\mathfrak{p}=\mathrm{id}_K$ for all $\mathfrak{p}\mid (p)$.
By the density theorem of Chebotarëv there are infinitely many such primes $p$. 
Let $m,n \in \mathbb{N}$ then the local $2$-function property reads
\begin{align}\label{eq: starting point}
a_{p^n m}-\mathrm{Frob}_p(a_{p^{n-1}m}) &= a_{p^n m} - a_{p^{n-1} m} \nonumber\\
&= \sum_{i=1}^r A_i \left(\alpha_i^{p^n m}-\alpha_i^{p^{n-1}m}\right)\nonumber\\
&\equiv 0\mod p^{2n}\mathcal{O}_p.
\end{align}

Before we dive into the proof, we give an intuition  of why \Cref{main theorem} is correct.
Since the congruence given in \cref{eq: starting point} is valid for infinitely many primes and all $m,n\in \mathbb{N}$, it should be true that these congruences already hold for each summand individually.
In other words, we expect
\begin{align*}
\alpha_i ^{p^nm} - \alpha_i^{p^{n-1}m} \equiv 0 \mod p^{2n} \mathcal{O}_p,
\end{align*}
for all $i\in\{1,...,r\}$ and all $m,n\in \mathbb{N}$ and all primes $p$ that split completely in $K|\mathbb{Q}$.
Therefore, we should be able to reduce \cref{eq: starting point} to the case $r=1$.
The case $r=1$ is subject of \Cref{lem: a sufficient criterion for root of unity}
Indeed, the speed of convergence of \cref{eq: starting point} is the crucial obstruction.

\begin{lemma}\label{lem: a sufficient criterion for root of unity}

Let $x \in K^\times$ and $p \in \mathbb{Z}$ a prime, which is unramified in $K|\mathbb{Q}$ and splits completely, such that $\iota_p(x)$ is a $p$-adic unit.
Suppose that
\begin{align*}
	\iota_p(x)^{p^n} - \iota_p(x)^{p^{n-1}} &\equiv 0 \mod p^{2n} \mathcal{O}_p,
\end{align*}
for all $n \in \mathbb{N}$.
Then $x$ is a root of unity in $K$.

\end{lemma}

\begin{proof}
If $p$ splits completely in $K | \mathbb{Q}$, then for all prime ideals $\mathfrak{p} \subset \mathcal{O}$ dividing $(p)$ we have $K_\mathfrak{p} \cong \mathbb{Q}_p$ and $\mathcal{O}_\mathfrak{p}\cong \mathbb{Z}_p$.
Let $\overline x \in \mathbb{Q}_p$ denote the image of $\iota_\mathfrak{p}(x)$ under this identification.
Then we have in particular $\overline x \in \mathbb{Z}_p$ and the congruence assumption reformulates to
\begin{align*}
 \overline x^{p^n} - \overline x^{p^{n-1}} \equiv 0 \mod p^{2n} \mathbb{Z}_p \quad \text{for all $n\in \mathbb{N}$}.
\end{align*}
Equivalently,
\begin{align*}
\overline x^{p^{n-1}(p-1)}\equiv 1\mod p^{2n} \mathbb{Z}_p \quad \text{for all $n \in \mathbb{N}$}.
\end{align*}
Recall that the Iwasawa logarithm preserves the $p$-adic order, therefore
\begin{align*}
p^{n-1}\log_p( \overline x^{p-1})\equiv 0\mod p^{2n}\mathbb{Z}_p \quad \text{for all $n \in \mathbb{N}$}.
\end{align*}
Hence, $\log_p(\overline x^{p-1})\equiv 0\mod p^{n+1}$ for all $n \in \mathbb{N}$, implying $\overline x \in \ker \log_p$.
Since $\iota_p(x)$ is a $p$-adic unit, $\overline x$ is a root of unity in $\mathbb{Z}_p$ and consequently, $x$ needs to be a root of unity in $K$.
\end{proof}

The obvious problem is that, \textit{a priori}, one may not take any conclusions on the $p$-divisibility of the sumands in \cref{eq: starting point} by only knowing the $p$-divisibility of the hole sum.
This is reflected by the fact that $\log_p$ is not additive.
That makes it unlikely to generalize the procedure in the proof of \Cref{lem: a sufficient criterion for root of unity} to \cref{eq: starting point} for $r>1$.
Hence, there does not seem to exist a true reduction of \cref{eq: starting point} to the case $r=1$.
At the other hand, \Cref{lem: a sufficient criterion for root of unity} surprisingly suggests  that it should be sufficient to investigate the $2$-function property \cref{eq: starting point} for only one suitably chosen prime $p$ (which is only possible since there are infinitely many such primes by Chebotarëv Density Theorem).
Therefore, the strategy we will pursue is a proof by contradiction:  We will assume that there is no root of unity among  $\alpha_i$, for $i=1,...,r$.
By \Cref{lem: a sufficient criterion for root of unity}, this amounts in saying, that the individual sumands $\alpha_i^{mp^n}- \alpha_i^{mp^{n-1}}$, for $i=1,...,r$, are converging \textit{slow} towards zero (they are converging after all by Euler's \Cref{Euler's theorem}).
For a suitable chosen prime (such that all relevant quantities are $p$-adic units), the $p$-adic estimations of the error functions $\rho_{i,n}(m) = \frac{\alpha_i^{mp^n}- \alpha_i^{mp^{n-1}}}{p^n}$ given by \Cref{prop: sebulba1} and \Cref{prop: sebulba2} in combination with the assumption given by \cref{eq: starting point} will then lead to a contradiction.
The resulting statement is given by \Cref{main theorem 2}.\\

Let $x \in \mathbb{Z}_p^\times$  and $m \in \mathbb{N}$.
By Euler's \Cref{Euler's theorem} there is a sequence $(\rho_n(m))_{n\in \mathbb{N}}\in \mathbb{Z}_p^{\mathbb{N}}$ such that
\begin{align*}
	x^{mp^n}-x^{mp^{n-1}}= p^{n}\rho_n(m).
\end{align*}
We also write $\kappa_n(m):= \mathrm{ord}_p(\rho_n(m)) \in \mathbb{N}_0 \cup \{\infty\}$.
As we will successively discover in \Cref{prop: sebulba1} and \Cref{prop: sebulba2}, $\kappa_n(m)$ is independent of $n,m \in \mathbb{N}$ for $\gcd (p,m)=1$.

\begin{proposition}\label{prop: sebulba1}
Let $p>2$ and $x \in \mathbb{Z}_p^\times$.
Then the sequence $\kappa_n(m) \in \mathbb{N}_0\cup \{\infty\}$ is independent of $n$, i.e. $\kappa_1(m)= \kappa_n(m)$ for all $n\in\mathbb{N}$.
If $\kappa(m):=\kappa_1(m) \neq \infty$, then
\begin{align*}
\rho_{n+1}(m)\equiv \rho_{n}(m)\mod p^{n+2 \kappa(m)} \mathbb{Z}_p
\end{align*}
for all $n \in \mathbb{N}$.
\end{proposition}

\begin{proof}
To simplify the notation, let $\rho_n:=\rho_n(m)$, $x=x^m$ and $\kappa_n= \kappa_n(m)$.
Suppose $\rho_{n_0}=0$ for some $n_0\in \mathbb{N}_0$.
But then, the equation $x^{p^{n_0-1}(p-1)}=1$ implies that $x$ is a root of unity in $\mathbb{Z}_p$ and therefore $x^{p^{n-1}(p-1)}=1$ for all $n\in \mathbb{N}$, i.e. $\rho_n=0$ for all $n \in \mathbb{N}$.
Recall that the set of torsion elements of $\mathbb{Z}_p$ (i.e. the set of roots of unity in $\mathbb{Q}_p$) are given by $\mu_{p-1}$, the set of $(p-1)$-th roots of unity.
Conversely, if $x$ is a root of unity, we therefore have $\rho_n=0$ for all $n\in \mathbb{N}$.
Suppose therefore, that $\rho_n \neq 0$  for all $n \in \mathbb{N}$, i.e. $x$ is not a root of unity in $\mathbb{Z}_p$.
Then the statement follows by using the \textit{Binomial Theorem}.
We have
\begin{align*}
	x^{p^n}
	&=\left(x^{p^{n-1}}\right)^p
	=\left(x^{p^n}-p^n\rho_n\right)^p
	=\sum_{k=0}^p\binom pkx^{kp^n}(-1)^{p-k}
		p^{n(p-k)}\rho_n^{p-k}\\
	&=
	x^{p^{n+1}}+\sum_{k=1}^{p-1}
	\binom pkx^{kp^n}(-1)^{p-k} p^{n(p-k)}\rho_n^{p-k}
		+(-1)^pp^{np}\rho_n^p.\\
	&=x^{p^{n+1}}+\sum_{k=1}^{p-1}\binom{p-1}kx^{kp^n}
	\frac{(-1)^{p-k}}{p-k}\rho_n^{p-k}p^{n(p-k)+1}
	+(-1)^p\rho_n^p p^{np}.
\end{align*}
Therefore, by using the definition $\rho_{n+1}= \frac1{p^{n+1}}\left( x^{p^{n+1}} - x^{p^n} \right)$, we obtain
\begin{align*}
p^{n+1}\rho_{n+1} &= \sum_{k=1}^{p - 1} \binom{p-1}k x^{kp^n} \frac{(-1)^{p-k+1}}{p-k} \rho_n^{ p - k } p^{ n ( p - k ) + 1 } - ( - 1 )^p \rho_n^p p^{ n p },\\
\Leftrightarrow\,\rho_{n+1} &= \sum_{k=1}^{p-1} \binom{p-1}k x^{ k p^n} \frac{ ( - 1 )^{ p - k + 1 }}{ p - k } \rho_n^{p-k} p^{ n ( p  - k - 1 )} - ( - 1 )^p \rho_n^p p^{ n ( p - 1 ) - 1 }.
\end{align*}
If $p>2$, we find modulo $p^{n+2 \kappa_n}$,
\begin{align*}
\rho_{n+1}
&\equiv x^{(p-1)p^n} \rho_n \mod p^{n+2\kappa_n}\mathbb{Z}_p.
\end{align*}
From this congruence it is evident, that $\kappa_{n+1}=\kappa_{n}$.
We therefore write $\kappa$ for $\kappa_n$.
Furthermore, using $x^{(p-1) p^n} = 1 + p^{n+1} x^{-p^n} \rho_{n+1}$ once more, this leads to
\begin{align*}
\rho_{n+1}&\equiv x^{(p-1)p^{n}}\rho_n\mod p^{n+2\kappa} \mathbb{Z}_p\\
&= \left( 1 + p^{n+1} x^{-p^n} \rho_{n+1} \right) \rho_n\\
&\equiv \rho_n \mod p^{n+2\kappa} \mathbb{Z}_p,
\end{align*}
which finishes the proof.
\end{proof}

\begin{proposition}\label{prop: sebulba2}
Let $x \in \mathbb{Z}_p^\times$ be a $p$-adic unit and $n,m\in\mathbb{N}$ be integers such that $\gcd (m,p) = 1$.
Then $\kappa:=\kappa(m) \in \mathbb{N}_0\cup \{\infty\}$ does not depend on $m$.
Furthermore, if $\kappa < \infty$, then
\begin{align*}
	\rho_n(m)\equiv m x^{(m-1)p^{n-1}}\rho_n(1)
	\mod p^{n+2\kappa}.
\end{align*}
\end{proposition}

\begin{proof}
Fix $n\in \mathbb{N}$.
If $x$ is a root of unity in $\mathbb{Z}_p$, then $\kappa(m)=\infty$ for all $m \in \mathbb{N}$.
Conversely, if $\rho_n(m)$ vanishes for some $m\in \mathbb{N}$, then $x$ is a root of unity in $\mathbb{Z}_p$. Therefore, let $x$ be not a root of unity in $\mathbb{Z}_p$.
Since $\rho_n(1)\neq 0$ we have
\begin{align*}
	x^{-(m-1)p^{n-1}}\frac{\rho_n(m)}{\rho_n(1)}
	&=\frac{1-x^{mp^{n-1}(p-1)}}{1-x^{p^{n-1}(p-1)}}
	=\sum_{k=0}^{m-1}x^{kp^{n-1}(p-1)}\\
	&= m + p^n \sum_{k=1}^{m-1} x^{-kp^{n-1}} \rho_n(k).
\end{align*}
The above computation shows that $\kappa(m )= \mathrm{ord}_p ( \rho_n ( m )$ is constant in $m$, since $\gcd (m,p) = 1$.
Therefore, write $\kappa:= \kappa(m)$ for all $m\in \mathbb{N}$.
In particular, the $p$-adic order of the sum $\sum_{k=1}^{m-1} x^{-kp^{n-1}} \rho_n(k)$ is at least $\kappa$ (since every single summand has $p$-adic order greater or equal to $\kappa$).
Therefore,
\begin{align*}
	\rho_n(m)\equiv m x^{(m-1)p^{n-1}}\rho_n(1) \mod p^{n+2\kappa} \mathbb{Z}_p,
\end{align*}
as stated.
\end{proof}

\begin{theorem}\label{main theorem 2}
Let $p\in \mathbb{Z}$ be an odd prime.
Let $r\in \mathbb{N}$ such that $r< p$ and for all $i =1,...r$ let $x_i, B_i \in \mathbb{Z}_p^\times$ such that $x_k \neq x_\ell \mod p \mathbb{Z}_p$ for $k\neq \ell$.
Suppose the validity of the following family of congruences
\begin{align}\label{2-function property for p}
\sum_{i=1}^r B_i\left( x_i^{mp^n}-x_i^{mp^{n-1}} \right) \equiv 0 \mod p^{2n} \mathbb{Z}_p \quad \text{for all $m,n \in \mathbb{N}$}.
\end{align}
Then $x_i$ is a root of unity in $\mathbb{Z}_p$ for all $i=1,...,r$.
\end{theorem}

\begin{proof}
Suppose there is a $j\in \{1,..., r\}$ such that $x_j$ is a root of unity in $\mathbb{Z}_p$.
Then $x_j^{p-1}=1$ and therefore
\begin{align*}
	x_j^{mp^n} - x_j^{mp^{n-1}} = x_j^{mp^{n-1}} \left( x_j^{(p-1)m} - 1 \right) = 0.
\end{align*}
Hence, \cref{2-function property for p} becomes a reduced sum with $r-1$ summands of the same type, namely,
\begin{align*}
\sum_{i=1}^r B_i \left( x_i^{m p^n} - x_i^{m p^{n-1}} \right) = \sum_{\substack{i=1\\ i\neq j}}^r B_i \left( x_i^{m p^n} - x_i^{m p^{n-1}} \right).
\end{align*}
Therefore, w. l. o. g. we may assume that none of the $x_i$ \cref{2-function property for p} are roots of unity.
We will lead this assumption to a contradiction, which then implies, that all $x_i$ are roots of unity in $\mathbb{Z}_p$.
In the following, we will write
\begin{align*}
	\rho_{i,n}(m)
	:= \frac1{p^n} \left(x_{i}^{mp^n}-x_{i}^{mp^{n-1}}\right),
	\qquad\text{and}\qquad
	\sigma_n(m)
	:=\frac1{p^n}\sum_{i=1}^r B_i\rho_{i,n}(m)
\end{align*}
for suitable $\rho_{i,n}(m),\sigma_n(m)\in \mathbb{Z}_p$.
Note that $\sigma_n(m)$ is indeed in $\mathbb{Z}_p$ by \cref{2-function property for p}.
In particular, we have $\rho_{i,n}(m)\neq 0$ for all $i,n,m\in \mathbb{N}$ with $\gcd(m,p)=1$.
By \Cref{prop: sebulba1} and \Cref{prop: sebulba2} we have for every $i=1,...,r$ a $\kappa_i \in \mathbb{N}$ such that $\kappa_i=\mathrm{ord}_p(\rho_{i,n}(m))$ for all $n,m\in \mathbb{N}$ with $\gcd(m,p) =1$.
Define $\kappa:= \min\{\kappa_i\,|\, i=1,...,r\}$.

Within this scenario, we will prove the following statement:\textit{
For all $n,m\in\mathbb{N}$ with $\gcd(m,p)=1$ we have
\begin{align}\label{chewbacca}
\sum_{i=1}^r B_ix_i^{(m-1)p^{n-1}}
\rho_{i,n}(1)\equiv 0\mod p^{n+2\kappa} \mathbb{Z}_p.
\end{align}}
By applying \Cref{prop: sebulba1} to each $\rho_{i,n+1}$ separately, we obtain
\begin{align*}
	p^{n+1}\sigma_{n+1}(m)
	=\sum_{i=1}^r B_i \rho_{i,n+1}(m)
	\equiv\sum_{i=1}^r B_i \rho_{i,n}(m)
	= p^n \sigma_{n} \mod p^{n+2\kappa }\mathbb{Z}_p.
\end{align*}
Dividing the above equation by $p^n$, we obtain
\begin{align}
	p \sigma_{n+1}(m) \equiv \sigma_n(m) \mod p^{2\kappa} \mathbb{Z}_p,
\end{align}
for all $n\in \mathbb{N}$.
Iteratively,
\begin{align*}
	\sigma_n(m) \equiv
	p^{2\kappa}\sigma_{n+2\kappa}(m)\equiv
	0\mod p^{2\kappa} \mathbb{Z}_p,
\end{align*}
for all $n\in\mathbb{N}$.
Therefore,
\begin{align}\label{Chewbacca}
	\sum_{i=1}^r B_i\rho_{i,n}(m)\equiv 0\mod 
	p^{n+2\kappa} \mathbb{Z}_p \qquad \text{for all $n\in\mathbb{N}$}.
\end{align}
From \cref{Chewbacca} the assertion \cref{chewbacca} for $m=1$ follows
immediately.
Now, let $m\in\mathbb{N}$ be arbitrary again.
Using \Cref{prop: sebulba2} gives
\begin{align}\label{star}
\sum_{i=1}^r B_i x_i^{\left( m - 1 \right) p^{ n - 1 } } \rho_{ i , n }(1) \equiv \frac1{m} \sum_{i=1}^r B_i \rho_{ i , n }( m ) \mod p^{ n + 2 \kappa } \mathbb{Z}_p.
\end{align}
Since $\gcd(m,p)=1$, applying \cref{Chewbacca} on the right-hand side of \cref{star} yields the formula \cref{chewbacca}.

Inserting $n=1$ and $m=1,...,r$ into \cref{chewbacca} yields the following system of linear equations, since $r<p$,
\begin{align*}
	\begin{pmatrix}
		B_1 & B_2 & \cdots & B_r\\
		B_1x_1& B_2 x_2& \cdots & B_r x_r\\
		\vdots&\vdots&\ddots&\vdots\\
		B_1 x_1^{r-1}& B_2 x_2^{r-1}&\cdots & B_rx_r^{r-1}
	\end{pmatrix}
	\begin{pmatrix}
		\rho_{1,1}(1)\\
		\rho_{2,1}(1)\\
		\vdots\\
		\rho_{r,1}(1)
	\end{pmatrix}\equiv
	0 \mod p^{1+2\kappa} \mathbb{Z}_p^r.
\end{align*}
The determinant of the above Vandermonde matrix is given by
\begin{align*}
	\det
	\begin{pmatrix}
		B_1 & B_2 & \cdots & B_r\\
		B_1x_1& B_2 x_2& \cdots & B_r x_r\\
		\vdots&\vdots&\ddots&\vdots\\
		B_1 x_1^{r-1}& B_2x_2^{r-1}&\cdots & B_r x_r^{r-1}
	\end{pmatrix}
	&=\big(\prod_{i=1}^r B_i \big) \times
	\prod_{1\leq k<\ell\leq r} (x_k-x_\ell)\\
	&\not \equiv 0 \mod p \mathbb{Z}_p.
\end{align*}
Hence, the determinant is invertible mod $p$, since $x_k - x_\ell$ is a $p$-adic unit for all $k\neq \ell$.
Consequently, $(\rho_{i,1}(1))_{i=1,...,r} \equiv 0 \mod p^{1+2\kappa}\mathbb{Z}^r_p$.
In other words, $\kappa \geq 1 + 2 \kappa$, which is the desired contradiction.
We conclude that all $x_1,..., x_r$  are roots of unity in $\mathbb{Z}_p$.
\end{proof}

\begin{corollary}\label{finish the puzzle}
Let $V\in \mathcal{S}_{\mathrm{rat}}^2 (K | \mathbb{Q})$, $V(z)\neq 0$, be the generating series of the underlying $2$-sequence $(a_n)_{n\in \mathbb{N}} = ([V(z)]_n)_{n\in \mathbb{N}}$, representing a rational function $F\in K(z)$.
Then there is an integer $N \in \mathbb{N}$ and coefficients $A_i \in \frac1N \mathbb{Z} \left[ D^{-1}\right]$ for $i=1,...,N$ such that
\begin{align*}
	F(z)=\sum_{i=1}^N\frac{A_i\zeta ^i z}{1- \zeta^i z},
\end{align*}
where $\zeta$ is a appropriate primitive $N$-th root of unity.

\end{corollary}

\begin{proof}
By \Cref{rational 1-functions} the coefficients $a_n = [V(z)]_n$ of $V$ are given by the power sums $a_n=\sum_{i=1}^r A_i \alpha_i^n$, for fixed $r\in\mathbb{N}$, where $A_i \in \mathbb{Q}^\times$ and where the $\alpha_i \in \overline{\mathbb{Q}}^\times$ are distinct algebraic numbers.
As mentioned at the beginning of this section, we may assume $\alpha_i\in K$ for all $i$.
Now, choose a prime $p\in \mathbb{Z}$ such that
\begin{enumerate}[(i)]
	\item
	$p$ is unramified in $K|\mathbb{Q}$ and splits completely,
	\item
	$\alpha_i$, $A_i$ and  $\alpha_k - \alpha_\ell$ are $p$-adic units for all $i=1,...,r$ and $k\neq \ell$,
	\item
	$\max\{r,2\}<p$.
\end{enumerate}
This choice of $p$ is possible by \Cref{chebotarev}.
Therefore, we have $K_\mathfrak{p}\cong\mathbb{Q}_p$ and $\mathcal{O}_\mathfrak{p} \cong \mathbb{Z}_p$ for all prime ideals $\mathfrak{p}\subset \mathcal{O}$ dividing $(p)$.
Hence, $\mathcal{O}_p$ may be identified with $\prod_{\mathfrak{p}|(p)} \mathbb{Z}_p$.
Since $\mathrm{Fr}_\mathfrak{p} = \mathrm{id}_K$, the local 2-function condition for $p$ then reads
\begin{align*}
	\sum_{i=1}^r A_i\left(
	\alpha_i^{mp^n}-\alpha_i^{mp^{n-1}}
	\right)\equiv 0 \mod p^{2n} \mathcal{O}_p,
\end{align*}
for all $m,n\in \mathbb{N}$.
For $B_i=\iota_\mathfrak{p}(A_i)$ and $x_i=\iota_\mathfrak{p}(\alpha_i)$ for $\mathfrak{p}\mid (p)$, \Cref{main theorem 2} states that $\alpha_i$ are all roots of unity.
In particular, the coefficients lie in a Galois subfield of $K$, which is abelian over $\mathbb{Q}$.
Choose an appropriate primitive $N$-th root of unity $\zeta$ and a bijection $\nu\colon \{1,...,N\} \rightarrow \{1,...,N\}$, such that $\alpha_i= \zeta^{\nu(i)}$ for all $i\in \{1,.., N\}$.
Without loss of generality, we may assume $A_i\in \mathbb{Q}$ (zeros are allowed) and
\begin{align*}
	a_n = \sum_{i=1}^N A_i \zeta^{in}.
\end{align*}
We observe that the coefficients $a_n$ are in $\mathcal{O} \left[ D^{-1} \right] \cap \mathbb{Q} ( \zeta )$.
By assumption, we have
\begin{align*}
	\begin{pmatrix}
	\zeta& \zeta^2& \cdots & \zeta^{N-1}& 1\\
	\zeta^2& \zeta^4& \cdots & \zeta^{2(N-1)}& 1\\
	\vdots &\vdots& \ddots& \vdots & \vdots\\
	\zeta^{N-1}& \zeta^{(N-1)2}&\cdots & \zeta^{(N-1)^2}& 1\\
	1&1&\cdots&1&1
	\end{pmatrix}
	\begin{pmatrix}
	A_1\\ A_2\\ \vdots\\ A_{N-1}\\ A_N
	\end{pmatrix}
	\in \mathcal{O}\left[D^{-1}\right]^N.
\end{align*}
The above matrix is invertible in $\mathbb{Q}(\zeta)$ with inverse
\begin{align*}
\left( \zeta^{ij}\right)_{\substack{i=1,...,N\\ j=1,...,N}}^{-1}= \frac1N\cdot\left(\zeta^{-ij}\right)_{\substack{i=1,...,N\\ j=1,...,N}}.
\end{align*}
Therefore, $A_i \in \frac1N \mathcal{O} \left[ D^{-1} , \zeta \right] \cap \mathbb{Q} =\frac 1N \mathbb{Z} \left[ D^{-1} \right]$ for all $i=1,...,N$.
\end{proof}

An obvious consequence of \Cref{finish the puzzle} is that the coefficients of a given $V$ are \textit{periodic} (as defined below) with the periodicity being a positive integer $P_V$ dividing the number $N$.
What remains to show is the simple fact, that a minimal such $N$ is given by $P_V$.
This is the statement of \Cref{prop: periodicity}.

\begin{definition}[Periodicity]\label{defi: periodicity}
Let $V \in \overline{\mathcal{S}}_{\mathrm{rat}}^2(K|\mathbb{Q})_{\mathrm{fin}}\setminus \{0\}$.
The \textit{periodicity} $P_V$ of $V$ is given by the periodicity of the coefficients of its Maclaurin series.
More precisely, $P_V\in \mathbb{N}$ is given by
\begin{align*}
	P_V = \min\{N\in \mathbb{N}\,|\, [V(z)]_n = [V(z)]_{n+N}\text{ for all $n\in \mathbb{N}$}\}.
\end{align*}
Note, that the existence of $P_V$ is ensured by \Cref{finish the puzzle}.
Furthermore,
\begin{align*}
\pi \colon \overline{\mathcal{S}}^2_{\mathrm{rat}}( K | \mathbb{Q} )_{\mathrm{fin}} \rightarrow \mathbb{N}_0
\end{align*}
denotes the map given by $\pi(V) = P_V$, if $V\neq 0$, and $\pi(0)=0$.
For $N\in \mathbb{N}_0$, we denote by $\overline{\mathcal{S}}_N$ ($\mathcal{S}_N$, resp.) the the preimage of $N$ under $\pi$ (the intersection of the preimage of $N$ under $\pi$ and $\mathcal{S}_{\mathrm{rat}}^2 (K|\mathbb{Q})_{\mathrm{fin}}$, resp.), i.e.
\begin{align*}
\overline{\mathcal{S}}_N := \pi^{-1}(N) \subset \overline{\mathcal{S}}_{\mathrm{rat}}^2 ( K | \mathbb{Q})_{\mathrm{fin}}
\quad\text{and}\quad
\mathcal{S}_N := \pi^{-1}(N) \cap \mathcal{S}_{\mathrm{rat}}^2(K|\mathbb{Q})_{\mathrm{fin}}.
\end{align*}
\end{definition}

\begin{proposition}\label{prop: periodicity}
Let $V \in \overline{\mathcal{S}}_{\mathrm{rat}}^2(K|\mathbb{Q})_{\mathrm{fin}}\setminus\{0\}$ and let $P_V= \pi(V)$ be the periodicity of $V$, i.e. $V\in \overline{S}_N$.
Then $[V(z)]_n\in K \cap \mathbb{Q}(\zeta_{P_V})$ and, for an appropriate $P_V$-th primitive root of unity $\zeta_{P_V}$, there are $A_j\in \mathbb{Q}$, $1\leq j \leq P_V$, such that
\begin{align*}
A_1 \neq 0
\quad\text{and}\quad
V(z)= \frac{z}{1-z^{P_V}}\sum_{i=0}^{P_V-1} a_{i+1}z^i, \quad \text{and} \quad a_i=\sum_{j=1}^{P_V} A_j \zeta^{ij}.
\end{align*}
Furthermore, the map $\pi \colon \mathcal{S}^2_{\mathrm{rat}}( K | \mathbb{Q} ) \rightarrow \mathbb{N}_0$ is surjective.
\end{proposition}

\begin{proof}
By \Cref{main theorem}, there is an $N \in \mathbb{N}$ and a primitive $N$-th root of unity and suitable coefficients $A_i$, $1\leq i\leq N$ such that
\begin{align}\label{eq: again main theorem}
V(z)= \sum_{i=1}^N \frac{A_i\zeta^i z}{1-\zeta^iz}.
\end{align}
From this representation of $V$, one immediately sees $P_V \leq N$.
Let $\hat N$ denote the minimum of the set of all $N\in \mathbb{N}$, such that $V$ permits a representation given by \cref{eq: again main theorem}.
In particular, $\hat N \leq P_V$.
Assume that all $A_i$ (which depend on $\hat N$) with $\gcd(i,\hat N)=1$ are vanishing.
In that case, $V$ has an representation given by \cref{eq: again main theorem} with $N \leq \hat N$, violating the minimality of $\hat N$.
Therefore, we may assume, that at least one $A_i$ does not vanish for $\gcd (i, \hat N)=1$.
This implies by the representation given in \cref{eq: again main theorem}, that $V$ has a pole of order $1$ at an $\hat N$-th primitive root of unity, say $\zeta_{\hat N}$.
Therefore, we may assume $A_1\neq 0$.
Since $V$ has periodicity $P_V$, we can write
\begin{align*}
V(z) = \sum_{i=1}^\infty a_i z^i = \sum_{i=1}^{P_V} a_i \sum_{k=0}^\infty z^{i + P_V k}
= \sum_{i=1}^{P_V} a_i z^i \sum_{k=0}^\infty z^{P_V k}= \frac{z}{1-z^{P_V}} \sum_{i=0}^{P_V -1}a_{i+1}z^i.
\end{align*}
This shows, that all singularities of $V$ are roots of the polynomial $1-z^{P_V}$.
Therefore, $\zeta_{\hat N}$ is a $P_V$-th root of unity and hence, $\hat N \leq P_V$.
We conclude $\hat N = P_V$.
For the surjectivity of $\pi$ recall the map $\varepsilon^{(2)}_k\colon \mathcal{S}^2(K|\mathbb{Q}) \rightarrow \mathcal{S}^2(K|\mathbb{Q})$ for $k\in \mathbb{N}$ from \cref{eq: shift}.
Now, take $N\in \mathbb{N}$, and let $V(z)=\frac{z}{1-z} = \sum_{k=1}^\infty z^k \in \mathcal{S}^2{\mathrm{rat}}(K|\mathbb{Q})$.
Then, in particular, $\varepsilon^{(2)}_N V \in \mathcal{S}^2_{\mathrm{rat}} (K|\mathbb{Q})$, and
\begin{align*}
\pi \left( \varepsilon_N^{(2)}\left(V(z)\right)\right)= \pi\left( \frac{N^2 z^N}{1-z^N}\right) =N.
\end{align*}
Hence, $\pi$ is surjective.
\end{proof}

\Cref{main theorem (2)} now follows from \Cref{finish the puzzle} and \Cref{prop: periodicity}.

\bibliographystyle{amsplain}

\end{document}